\documentclass[11pt]{amsart}

\usepackage{kotex}
\usepackage{amsmath, amssymb, amsthm, mathtools}
\usepackage{enumitem}
\usepackage{geometry}
\usepackage{float}
\usepackage[hidelinks]{hyperref}

\usepackage{graphicx}
\usepackage{tikz}
\usetikzlibrary{arrows.meta, positioning, shapes.geometric, calc}

\theoremstyle{plain}
\newtheorem{theorem}{Theorem}[section]
\newtheorem{lemma}[theorem]{Lemma}
\newtheorem{proposition}[theorem]{Proposition}
\newtheorem{corollary}[theorem]{Corollary}

\theoremstyle{definition}
\newtheorem{definition}[theorem]{Definition}
\newtheorem{example}[theorem]{Example}
\newtheorem{question}{Question}[section]

\theoremstyle{remark}
\newtheorem{remark}[theorem]{Remark}

\newcommand{\Sp}{\mathcal{S}_P}
\newcommand{\Dp}{\mathcal{D}_P}
\newcommand{\F}{\mathcal{F}}
\newcommand{\N}{\mathbb{N}}
\newcommand{\R}{\mathbb{R}}
\newcommand{\e}{\epsilon}
\DeclareMathOperator{\Range}{Range}
\newcommand{\cl}[1]{\overline{#1}}

\title[Minimal collections of sequences]{On minimal collections of sequences for testing continuity}
\author{Gyuhyun Lim}

\address{Department of Mathematics Education, College of Education,
Seoul National University, Building 10, Room 204,
1 Gwanak-ro, Gwanak-gu, Seoul 08826, Republic of Korea}

\email{gyuhyuny@snu.ac.kr}
\date{May 2026}

\begin{document}

\begin{abstract}
We study \emph{test sets}: subfamilies of sequences converging to a point $P$ that still suffice to detect every discontinuity of real-valued functions at $P$.
Ordered by inclusion, these test sets form a poset.
Under natural hypotheses at $P$, we prove that this poset has a minimal element.
We also analyze its maximal chains, showing that some have a least element, while others do not.
Finally, on the \emph{sequential fan} we give a concrete realization in which the minimal test set produced by our construction has strictly smaller cardinality than the full family of convergent sequences.
\end{abstract}

\maketitle
\tableofcontents

\section{Introduction}\label{sec:intro}

In ordinary freshman calculus and introductory analysis, continuity at a point is often tested using sequences:
to verify that a function is continuous at $P$, one checks its behavior along every sequence converging to $P$.
While this criterion is correct in many familiar settings, it raises a natural efficiency question:
must one really inspect the entire family of convergent sequences, or can a much smaller subfamily already detect every possible failure of continuity at $P$?

A familiar motivation comes from primality testing.
The statement ``$n$ is prime'' is equivalent to ``$n$ is not divisible by any integer $d$ with $1<d<n$.'' Nevertheless, an equivalent formulation need not be an \emph{optimal} procedure.
In practice, it suffices to test divisibility only for prime $d\le \sqrt{n}$, and this still yields a logically equivalent criterion with a drastically smaller search space.
The guiding idea of this paper is that sequential tests for continuity can exhibit the same phenomenon:
one may be able to discard many sequences while still detecting \emph{every} discontinuity.

A first concrete hint appears on the real line.
For $f:\R\to\R$, continuity at a point $P$ can be decided by comparing two one-sided behaviors:
if the left limit and the right limit exist and both are equal to $f(P)$, then $f$ is continuous at $P$.
Equivalently, it suffices to test $f(T_n)\to f(P)$ only for sequences that approach $P$ eventually from the left,
and for sequences that approach $P$ eventually from the right.
In this reformulation, highly oscillatory sequences are not examined directly, yet no discontinuity is missed.

The purpose of this paper is to study this efficiency problem in an abstract topological setting.
Rather than asking only whether the full family of convergent sequences can be reduced, we ask which smaller collections of convergent sequences still suffice to detect every failure of continuity at $P$.
To make this question precise, one is naturally led to look not only at all sequences converging to $P$, but also at the smaller subfamilies along which a given discontinuous function actually reveals its discontinuity.
The problem then becomes one of finding compressed testing families that still meet every such witness family.
These notions will be formalized in Section~\ref{sec:prelim}; see in particular Definitions~\ref{def:SpDpDf} and \ref{def:testset}.

Once those testing families are available, they may be compared by inclusion.
This is the natural order for the present problem:
passing from a larger testing family to a smaller one corresponds to discarding sequences while retaining the ability to detect every discontinuity.
The resulting inclusion structure, formalized as the poset $(\F,\subseteq)$ in Definition~\ref{def:testset}, is the main object studied in this paper.

Two structural questions drive the paper.

\begin{question}\label{q:minimal_exists}
Does the inclusion-ordered family of test sets contain a minimal element?
\end{question}

\begin{question}\label{q:chains_reflect}
If minimal test sets exist, must they be reflected in every maximal chain of test sets?
\end{question}

Question~\ref{q:minimal_exists} asks whether there is an irreducible test set.
Question~\ref{q:chains_reflect} asks whether such minimal objects are visible uniformly across the order structure of $(\F,\subseteq)$.
These questions are answered in Theorems~\ref{thm:minimal_exists_intro} and \ref{thm:chains_exist_intro}.

At the same time, the original motivation is one of efficiency.
A test set that is minimal with respect to inclusion need not have smaller cardinality than the full family $\Sp$.
Thus, throughout the paper, we distinguish inclusion-minimality from a genuine reduction in cardinality.
The general results establish the existence of inclusion-minimal test sets, while the sequential fan considered later shows that such a reduction in cardinality can also occur.

\medskip
\noindent\textbf{Standing framework.}
All arguments are carried out at a fixed point $P$ under four standing assumptions:
\begin{itemize}[leftmargin=2.2em,itemsep=0.6em]
\item $(X,P)$ is \emph{Fr\'echet--Urysohn at $P$} (Definition~\ref{def:FU_at_P}), so discontinuity at $P$ is detectable by sequences; see Lemma~\ref{lem:seq_char_continuity}.

\item $X$ is $T_1$ at $P$, so the later set-theoretic construction can be carried out inside $X\setminus\{P\}$; see Lemma~\ref{lem:kernel_T1} and Subsection~\ref{subsec:nonT1_spaces}.

\item $P$ is \emph{non-isolated}, ensuring that there are genuinely nontrivial sequences approaching $P$.

\item There exists a real-valued function discontinuous at $P$, excluding the degenerate case in which every testing family is automatically sufficient; see Lemma~\ref{lem:isolated_implies_Dp_empty} and Remark~\ref{rem:vacuous_case}.
\end{itemize}

\subsection*{Main Results}

\begin{theorem}[Existence of minimal test sets]\label{thm:minimal_exists_intro}
Assume the standing framework at $P$ described above.
Then $(\F,\subseteq)$ has a minimal element.
\end{theorem}

\begin{theorem}[Chains need not reflect minimal test sets]\label{thm:chains_exist_intro}
Assume the standing framework at $P$ described above.
Then there exist maximal chains $\mathcal{C}_{\mathrm{bad}},\mathcal{C}_{\mathrm{good}}\subseteq\F$ such that
\[
\bigcap \mathcal{C}_{\mathrm{bad}} \notin \F,
\qquad
\bigcap \mathcal{C}_{\mathrm{good}} \in \F.
\]
\end{theorem}

The first theorem gives an irreducible testing family.
The second shows that this phenomenon is not reflected uniformly in the chain structure of $(\F,\subseteq)$:
some maximal chains admit a least element, while others do not.
In particular, once a minimal test set lies in a chain, it should be realized as the intersection of the chain; see Lemma~\ref{lem:min_chain_intersection}.

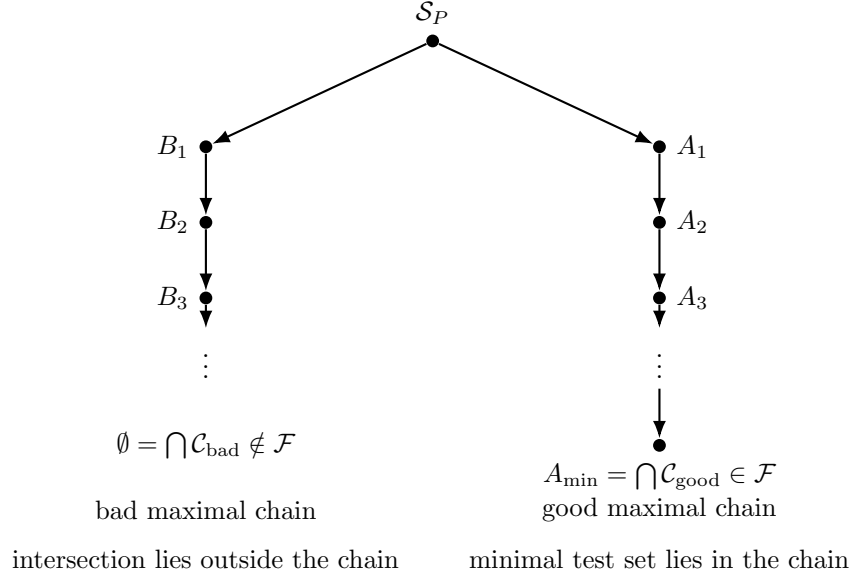
\begin{figure}[H]
\centering
\begin{tikzpicture}[
  filled/.style={circle, fill=black, inner sep=1.7pt},
  lab/.style={font=\small},
  arr/.style={-Latex, thick},
  every node/.style={align=center}
]

\node[filled, label={[lab]above:$\Sp$}] (S) at (0,4.4) {};

\node[filled, label={[lab]left:$B_1$}] (B1) at (-3.0,3.0) {};
\node[filled, label={[lab]left:$B_2$}] (B2) at (-3.0,2.0) {};
\node[filled, label={[lab]left:$B_3$}] (B3) at (-3.0,1.0) {};
\node[lab] (Bdots) at (-3.0,0.2) {$\vdots$};
\node[lab] (BadCap) at (-3.0,-0.95) {$\emptyset=\bigcap\mathcal C_{\mathrm{bad}}\notin\F$};

\node[filled, label={[lab]right:$A_1$}] (A1) at (3.0,3.0) {};
\node[filled, label={[lab]right:$A_2$}] (A2) at (3.0,2.0) {};
\node[filled, label={[lab]right:$A_3$}] (A3) at (3.0,1.0) {};
\node[lab] (Adots) at (3.0,0.2) {$\vdots$};
\node[filled, label={[lab]below:$A_{\min}=\bigcap\mathcal C_{\mathrm{good}}\in\F$}] (Amin) at (3.0,-0.95) {};

\draw[arr] (S) -- (B1);
\draw[arr] (S) -- (A1);

\draw[arr] (B1) -- (B2);
\draw[arr] (B2) -- (B3);
\draw[arr] (B3) -- (Bdots);

\draw[arr] (A1) -- (A2);
\draw[arr] (A2) -- (A3);
\draw[arr] (A3) -- (Adots);
\draw[arr] (Adots) -- (Amin);

\node[lab] at (-3.0,-1.8) {bad maximal chain};
\node[lab] at (3.0,-1.8) {good maximal chain};

\node[lab] at (-3.0,-2.45) {intersection lies outside the chain};
\node[lab] at (3.0,-2.45) {minimal test set lies in the chain};

\end{tikzpicture}
\caption{A schematic picture of the two chain behaviors established in this paper.
On the left, a bad maximal chain does not induce a minimal test set.
On the right, a good maximal chain has intersection equal to a minimal test set.}
\label{fig:intro_schematic}
\end{figure}

\subsection*{Outline of the Paper}

In Section~\ref{sec:prelim}, we introduce the common sequential language used throughout the paper and formalize the notions of witness families, test sets, and maximal chains.

In Section~\ref{sec:bad}, we construct a bad maximal chain.
The key tool is the invariance of witnessing under finite modification:
by adjusting finitely many initial terms of a witnessing sequence to agree with a fixed reference sequence, we obtain a descending chain of test sets whose intersection is empty.
Extending this chain to a maximal one yields the bad behavior asserted in Theorem~\ref{thm:chains_exist_intro}.

In Section~\ref{sec:minimal}, we construct a minimal test set to prove Theorem~\ref{thm:minimal_exists_intro}.
The existence of a good maximal chain then follows as a corollary, completing the proof of Theorem~\ref{thm:chains_exist_intro}.
We conclude the section by showing that, in the \emph{sequential fan}, the construction yields a minimal test set with strictly smaller cardinality than the full family $\Sp$.

\section{From continuity to test sets}\label{sec:prelim}

The purpose of this section is to recast the continuity problem at $P$ in a form that can be studied combinatorially.
Starting from the usual sequential criterion for continuity, we pass to \emph{witness families} of convergent sequences and then to the \emph{test sets} that meet all such families.
This produces the inclusion-ordered family $(\F,\subseteq)$ that will be studied in the rest of the paper.

\subsection{Standing framework at \texorpdfstring{$P$}{P}}

Throughout, $X$ is a topological space, $P\in X$ is fixed, and $\mathcal N(P)$ denotes the collection of open neighborhoods of $P$.
The terms \(T_1\), closure, neighborhood, and convergence are used in the standard sense of general topology; see Engelking~\cite{Eng89}.

\begin{definition}\label{def:FU_at_P}
We say that $(X,P)$ is \emph{Fr\'echet--Urysohn at $P$} if for every set $A\subseteq X$,
\[
P\in\cl{A}\ \Longrightarrow\ \exists (a_n)_{n\in\N}\subseteq A\ \text{with}\ a_n\to P.
\]
\end{definition}

\begin{definition}\label{def:nonisolated}
We say that $P$ is \emph{non-isolated} if every neighborhood $U$ of $P$ contains a point $x\neq P$.
Equivalently,
\[
P\in\cl{X\setminus\{P\}}.
\]
\end{definition}

One auxiliary notion will be useful later when we separate the point $P$ from the rest of the space.
Namely, we record the points that lie in every neighborhood of $P$.

\begin{definition}[Neighborhood kernel]\label{def:kernel}
Define the \emph{kernel} of $P$ by
\[
K_P:=\bigcap_{U\in\mathcal N(P)}U.
\]
\end{definition}

These assumptions have distinct roles in the later constructions.
The Fr\'echet--Urysohn hypothesis is classical in the study of spaces where sequences suffice \cite{Fra65}.
It is what makes a sequence-based formulation possible in the first place:
if discontinuity is present at $P$, then some sequence converging to $P$ will witness it; see Lemma~\ref{lem:seq_char_continuity}.
The non-isolation assumption ensures that there are genuinely nontrivial sequences approaching $P$.
The kernel $K_P$ records the points that are forced into every neighborhood of $P$.
Under the $T_1$ hypothesis at $P$, this kernel collapses to $\{P\}$.

\begin{lemma}\label{lem:kernel_T1}
Assume $X$ is $T_1$ at $P$.
Then $K_P=\{P\}$.

\end{lemma}

\begin{proof}
Since $P\in U$ for every $U\in\mathcal N(P)$, one always has $P\in K_P$.
Now fix $x\neq P$.
Because $X$ is $T_1$ at $P$, there exists an open neighborhood $U\in\mathcal N(P)$ such that $x\notin U$.
Hence $x\notin \bigcap_{V\in\mathcal N(P)}V=K_P$.
Therefore $K_P=\{P\}$.
\end{proof}

\begin{remark}\label{rem:kernel_consequence}
Under Lemma~\ref{lem:kernel_T1}, every point $x\neq P$ is excluded by some neighborhood of $P$.
Consequently, if a sequence converges to $P$, then any fixed value $x\neq P$ can occur only finitely many times: see Subsection~\ref{subsec:nonT1_spaces} for the case when $T_1$ fails at $P$.
\end{remark}

\subsection{Witnessing discontinuity and test sets}

We now introduce the objects that translate continuity at $P$ into a problem about families of sequences convergent to $P$.
The first step is to record, for each discontinuous function, exactly which convergent sequences witness its failure of continuity.

\begin{definition}[$\Sp$, $\Dp$, and $D(f)$]\label{def:SpDpDf}
\mbox{}
\begin{enumerate}[label=(\roman*)]
\item The family of sequences converging to $P$:
\[
\Sp:=\left\{T=(T_n)_{n\in\N}\ \middle|\ T_n\in X,\ T_n\to P\right\}.
\]

\item The class of real-valued functions discontinuous at $P$:
\[
\Dp:=\{f:X\to\R \mid f \text{ is discontinuous at } P\}.
\]

\item For a function $f:X\to\R$ in $\Dp$, the family of \emph{witnessing sequences}:
\[
D(f):=\left\{T\in\Sp \mid f(T_n)\not\to f(P)\right\}.
\]
\end{enumerate}
\end{definition}

Thus $\Sp$ is the full sequential testing family at $P$, while $D(f)\subseteq\Sp$ singles out those sequences along which the discontinuity of $f$ is actually visible.
The point of this language is that discontinuity is no longer treated only as a local topological failure at $P$;
it is now encoded by a concrete subfamily of the convergent sequences to $P$.

\begin{example}\label{ex:witnessing_in_R}
Let $X=\R$ with the usual topology and $P=0$.
Define
\[
f(x)=
\begin{cases}
\sin(1/x), & x\neq 0,\\
0, & x=0.
\end{cases}
\]
Then the sequence
\[
T_n=\frac{1}{2n\pi+\pi/2}
\]
belongs to $D(f)$, since $T_n\to 0$ but $f(T_n)=1$ for all $n$.
By contrast, the sequence
\[
S_n=\frac{1}{n\pi}
\]
does not belong to $D(f)$, because $S_n\to 0$ and $f(S_n)=0$ for all $n$.
Thus even for a fixed discontinuous function, some convergent sequences witness the discontinuity while others do not.
\end{example}

The next lemma is the basic translation principle of the paper.
Under the Fr\'echet--Urysohn hypothesis, continuity at $P$ is completely determined by the behavior of a function on the family $\Sp$.

\begin{lemma}[Sequential characterization of continuity at $P$]\label{lem:seq_char_continuity}
Assume $(X,P)$ is Fr\'echet--Urysohn at $P$.
For any function $f:X\to\R$, the following are equivalent:
\begin{enumerate}[label=(\roman*)]
\item $f$ is continuous at $P$;
\item for every $T\in\Sp$, one has $f(T_n)\to f(P)$.
\end{enumerate}
In particular,
\[
f\in\Dp \quad\Longleftrightarrow\quad D(f)\neq\emptyset.
\]
\end{lemma}

\begin{proof}
The implication (i)$\Rightarrow$(ii) holds in every topological space.

For (ii)$\Rightarrow$(i), assume that $f$ is not continuous at $P$.
Then there exists $\e>0$ such that for every neighborhood $U$ of $P$ there is $x\in U$ with
\[
|f(x)-f(P)|\ge \e.
\]
Let
\[
A:=\{x\in X:\ |f(x)-f(P)|\ge\e\}.
\]
Then $P\in\cl{A}$.
Since $(X,P)$ is Fr\'echet--Urysohn at $P$, there exists a sequence $T=(a_n)_{n\in\N}\subseteq A$ with $a_n\to P$.
Since
\[
|f(T_n)-f(P)|\ge\e \qquad\text{for all } n,
\]
it follows that $f(T_n)\not\to f(P)$.
This contradicts (ii).
Hence $f$ is continuous at $P$.
\end{proof}

The following lemma identifies exactly when the testing problem becomes degenerate.

\begin{lemma}\label{lem:isolated_implies_Dp_empty}
The following are equivalent:
\begin{enumerate}[label=(\roman*)]
\item $P$ is isolated;
\item $\Dp=\emptyset$.
\end{enumerate}
\end{lemma}

\begin{proof}
Assume first that $P$ is isolated.
Then $\{P\}\in\mathcal N(P)$.
Let $f:X\to\R$ be arbitrary and let $V\subseteq\R$ be an open neighborhood of $f(P)$.
Since $P\in f^{-1}(V)$, we have
\[
\{P\}\subseteq f^{-1}(V).
\]
So $f$ is continuous at $P$.
Since $f$ was arbitrary, no function is discontinuous at $P$, and therefore $\Dp=\emptyset$.

Conversely, assume that $P$ is not isolated.
Define $h:X\to\R$ by
\[
h(x)=
\begin{cases}
1, & x=P,\\
0, & x\neq P.
\end{cases}
\]
Then $h(P)=1$.
Let
\[
V:=\left(\frac12,\frac32\right).
\]
Then $V$ is an open neighborhood of $h(P)$, and
\[
h^{-1}(V)=\{P\}.
\]
Since $P$ is not isolated, $\{P\}$ is not a neighborhood of $P$.
Thus $h^{-1}(V)$ is not a neighborhood of $P$, so $h$ is not continuous at $P$.
Hence $h\in\Dp$, and therefore $\Dp\neq\emptyset$.

Thus $P$ is isolated if and only if $\Dp=\emptyset$.
\end{proof}

\begin{remark}\label{rem:vacuous_case}
If $\Dp=\emptyset$, then the condition of being a test set becomes vacuous:
there are no discontinuous functions whose witness families must be met.
By Lemma~\ref{lem:isolated_implies_Dp_empty}, this happens exactly when $P$ is isolated.
Accordingly, all existence statements about minimal test sets exclude precisely this degenerate case.
\end{remark}

Once the witness families $D(f)$ are available, the original efficiency problem can be reformulated.
Rather than testing continuity against the full family $\Sp$, one asks whether a smaller subfamily can still meet every witness family.

\begin{definition}[Test set]\label{def:testset}
A subset $A\subseteq\Sp$ is called a \emph{test set} if
\[
\forall f\in\Dp,\quad A\cap D(f)\neq\emptyset.
\]
That is, $A$ must meet the witness family of every discontinuous function at $P$. We denote by
\[
\F:=\{A\subseteq\Sp:\ A \text{ is a test set}\}
\]
the family of all test sets.
\end{definition}

\begin{remark}\label{rem:Sp_is_testset}
The full family $\Sp$ is a test set.
Indeed, if $f\in\Dp$, then $D(f)\neq\emptyset$ by Lemma~\ref{lem:seq_char_continuity}, and of course $D(f)\subseteq\Sp$.
Hence $\Sp\cap D(f)\neq\emptyset$ for every $f\in\Dp$.
\end{remark}

A test set may therefore be viewed as a compressed continuity detector:
it may discard many convergent sequences, but it must still hit every witness family $D(f)$.

Once the poset $(\F,\subseteq)$ has been identified, two structural notions become central.
One concerns the existence of smallest possible test sets, and the other concerns the behavior of maximal chains inside the poset.

\begin{definition}[Minimal test set]\label{def:minimal_testset}
A test set $A\in\F$ is called \emph{minimal} if it is a minimal element of $(\F,\subseteq)$; that is,
\[
\forall B\in\F,\quad B\subseteq A \implies B=A.
\]
\end{definition}

\begin{definition}[Good and bad maximal chains]\label{def:goodbad}
Let $\mathcal C\subseteq\F$ be a maximal chain in the poset $(\F,\subseteq)$, and write
\[
\bigcap\mathcal C:=\bigcap_{A\in\mathcal C}A.
\]
The chain $\mathcal C$ is called \emph{good} if $\bigcap\mathcal C\in\F$, and \emph{bad} otherwise.
\end{definition}

\begin{lemma}\label{lem:min_chain_intersection}
Let $A_{\min}\in\F$ be a minimal element.
If $\mathcal C\subseteq\F$ is a chain with $A_{\min}\in\mathcal C$, then
\[
\bigcap_{A\in\mathcal C}A=A_{\min}.
\]
\end{lemma}

\begin{proof}
Let $A\in\mathcal C$.
Since $\mathcal C$ is a chain and $A_{\min}\in\mathcal C$, the sets $A$ and $A_{\min}$ are comparable under inclusion.
By minimality of $A_{\min}$ we must have $A_{\min}\subseteq A$.
Since this holds for every $A\in\mathcal C$, we obtain
\[
A_{\min}\subseteq \bigcap_{A\in\mathcal C}A.
\]
On the other hand, because $A_{\min}\in\mathcal C$, we also have
\[
\bigcap_{A\in\mathcal C}A\subseteq A_{\min}.
\]
Therefore
\[
\bigcap_{A\in\mathcal C}A=A_{\min}.
\]
\end{proof}
The following is a standard consequence of Zorn's lemma; for background on maximal principles in set theory, see Jech~\cite{Jec03}.
We include the proof for completeness.
\begin{lemma}[Hausdorff maximal principle]\label{lem:chain_extension}
Let $(P,\le)$ be a partially ordered set, and let $C_0\subseteq P$ be a chain.
Then there exists a maximal chain $C\subseteq P$ such that
\[
C_0\subseteq C.
\]
\end{lemma}

\begin{proof}
Let
\[
\mathfrak C:=\{D\subseteq P:\ D \text{ is a chain and } C_0\subseteq D\},
\]
partially ordered by inclusion.
Since \(C_0\) is itself a chain and \(C_0\subseteq C_0\), we have \(C_0\in\mathfrak C\).
Thus \(\mathfrak C\neq\emptyset\).

We verify that every chain in $\mathfrak C$ has an upper bound in $\mathfrak C$.
Let $\{D_i\}_{i\in I}$ be a chain in $\mathfrak C$, and set
\[
D:=\bigcup_{i\in I}D_i.
\]
Then clearly $C_0\subseteq D$.
It remains to show that $D$ is a chain in $P$.

Let $x,y\in D$.
Then $x\in D_i$ and $y\in D_j$ for some $i,j\in I$.
Since $\{D_i\}_{i\in I}$ is a chain under inclusion, either $D_i\subseteq D_j$ or $D_j\subseteq D_i$.
Hence $x$ and $y$ belong to a common chain, so they are comparable in $P$.
Therefore $D$ is a chain, and thus $D\in\mathfrak C$.

By Zorn's lemma, $\mathfrak C$ has a maximal element $C$.
Then $C$ is a maximal chain in $P$ containing $C_0$.
\end{proof}

These definitions separate the two main directions of the paper.
Section~\ref{sec:bad} shows that minimal test sets need not be realized in every maximal chain.
Section~\ref{sec:minimal} constructs a minimal test set and then derives a good maximal chain from it using Lemmas~\ref{lem:min_chain_intersection} and \ref{lem:chain_extension}.

\section{Constructing a bad maximal chain}\label{sec:bad}

This section proves the negative side of Question~\ref{q:chains_reflect}.
We show that even when minimal test sets exist, they need not be reflected in every maximal chain of $(\F,\subseteq)$.
To prove this, we construct a descending chain of test sets whose intersection is empty, and then extend it to a maximal chain.
Lemma~\ref{lem:finite_modification} is the key technical tool, because it allows us to prescribe finitely many initial terms of a sequence without destroying its witness property.

\subsection{Finite modification}

We begin by recording the basic fact that changing only finitely many terms of a convergent sequence does not affect either convergence to $P$ or the property of witnessing discontinuity.

\begin{lemma}[Finite modification lemma]\label{lem:finite_modification}
Let $T=(T_n)\in\Sp$ and let $S=(S_n)$ be a sequence in $X$.
Assume that there exists $N\in\N$ such that
\[
T_n=S_n \qquad\text{for all } n\ge N.
\]
Then:
\begin{enumerate}[label=(\roman*)]
\item $S\in\Sp$;
\item for every function $f:X\to\R$, one has
\[
T\in D(f)\quad\Longleftrightarrow\quad S\in D(f).
\]
\end{enumerate}
\end{lemma}

\begin{proof}
(i) Let $U\in\mathcal N(P)$.
Since $T\to P$, there exists $K$ such that $T_n\in U$ for all $n\ge K$.
If $n\ge \max\{N,K\}$, then $S_n=T_n\in U$.
Hence $S\to P$, so $S\in\Sp$.

(ii) Suppose first that $T\in D(f)$.
Then $f(T_n)\not\to f(P)$.
Hence there exists $\e>0$ such that for every $k\in\N$ there exists $m\ge k$ with
\[
|f(T_m)-f(P)|\ge \e.
\]
Fix $k$ and choose such an $m\ge \max\{k,N\}$.
Then $S_m=T_m$, so
\[
|f(S_m)-f(P)|=|f(T_m)-f(P)|\ge \e.
\]
Thus $f(S_n)\not\to f(P)$, hence $S\in D(f)$.
The converse implication is proved in exactly the same way.
\end{proof}

\subsection{Prefix-fixed families}

Once finite modification is available, the natural way to build a descending chain is to prescribe longer and longer initial segments.
This leads to the prefix-fixed families introduced below.

Fix once and for all a reference sequence
\[
a=(a_k)_{k\in\N}\in\Sp.
\]

\begin{definition}[$A_n(a)$ and $B_n(a)$]\label{def:AnBn}
For each $n\in\N$, define
\[
A_n(a):=\Bigl\{T=(T_k)_{k\in\N}\in\Sp \ \Big|\ T_k=a_k \text{ for every } 1\le k\le n\Bigr\},
\]
and
\[
B_n(a):=A_n(a)\setminus\{a\}.
\]
When no confusion is likely, we simply write $A_n$ and $B_n$; see Figure~\ref{fig:badchain_prefix_picture}.
\end{definition}

The family $A_n(a)$ consists of sequences whose first $n$-terms are prescribed following $a$.
The family $B_n(a)$ is obtained by removing the single sequence $a$ itself.

\begin{figure}[ht!]
\centering
\begin{tikzpicture}[
  box/.style={draw, rounded corners, align=center, inner sep=4pt, font=\small},
  arr/.style={-Latex, thick},
  node distance=7mm and 10mm
]
\node[box] (Sp) {$\Sp$};
\node[box, below=of Sp] (A1) {$A_1(a)$\\(fix first term)};
\node[box, below=of A1] (A2) {$A_2(a)$\\(fix first two terms)};
\node[box, below=of A2] (A3) {$A_3(a)$\\(fix first three terms)};
\node[box, below=of A3] (dots) {$\vdots$};
\node[box, below=of dots] (cap) {$\bigcap_{n\in\N} A_n(a)=\{a\}$};

\draw[arr] (Sp) -- (A1);
\draw[arr] (A1) -- (A2);
\draw[arr] (A2) -- (A3);
\draw[arr] (A3) -- (dots);
\draw[arr] (dots) -- (cap);

\node[box, right=3.1cm of A1] (B1) {$B_1(a)=A_1(a)\setminus\{a\}$};
\node[box, below=of B1] (B2) {$B_2(a)=A_2(a)\setminus\{a\}$};
\node[box, below=of B2] (B3) {$B_3(a)=A_3(a)\setminus\{a\}$};
\node[box, below=of B3] (dots2) {$\vdots$};
\node[box, below=of dots2] (cap2) {$\bigcap_{n\in\N} B_n(a)=\emptyset$};

\draw[arr] (B1) -- (B2);
\draw[arr] (B2) -- (B3);
\draw[arr] (B3) -- (dots2);
\draw[arr] (dots2) -- (cap2);

\draw[arr] (A1) -- (B1);
\draw[arr] (A2) -- (B2);
\draw[arr] (A3) -- (B3);

\end{tikzpicture}
\caption{Prefix-fixing families $A_n(a)$ shrink to $\{a\}$, while removing $a$ forces $\bigcap_n B_n(a)=\emptyset$.}
\label{fig:badchain_prefix_picture}
\end{figure}
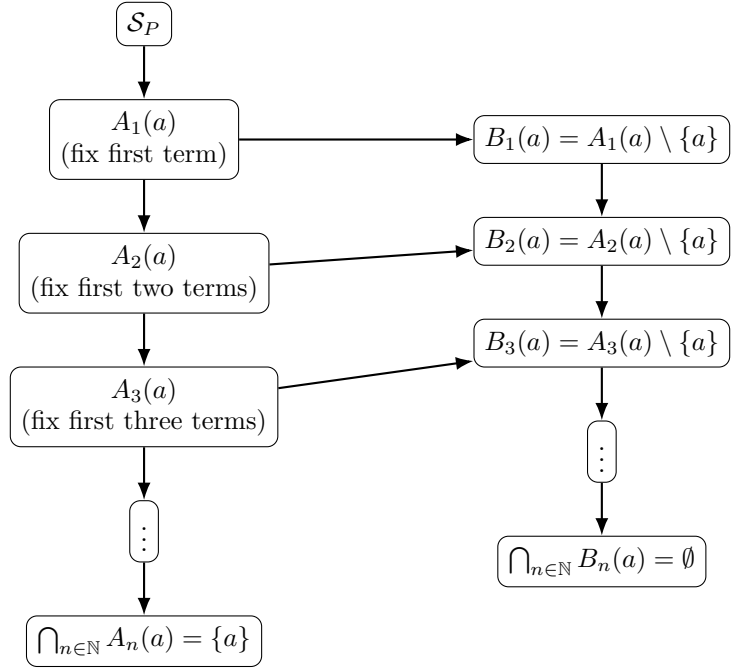

\begin{remark}\label{rem:nontrivial_space_from_Dp}
If \(\Dp\neq\emptyset\), then \(X\) is not a singleton.
Indeed, by Lemma~\ref{lem:isolated_implies_Dp_empty}, the assumption \(\Dp\neq\emptyset\) implies that \(P\) is non-isolated.
If \(X=\{P\}\), then \(P\) is isolated, a contradiction.
\end{remark}

At this point two facts must be checked.
First, each prefix-fixed family still has enough witnessing sequences to remain a test set.
Second, the families shrink so tightly that their total intersection disappears.

\begin{lemma}\label{lem:Bn_chain}
Assume $\Dp\neq\emptyset$.
Then for every $n\in\N$:
\begin{enumerate}[label=(\roman*)]
\item $B_n\neq\emptyset$;
\item $B_{n+1}\subseteq B_n$.
\end{enumerate}
\end{lemma}

\begin{proof}
The inclusion $B_{n+1}\subseteq B_n$ is immediate from $A_{n+1}\subseteq A_n$.

For nonemptiness, fix $n$.
By Lemma~\ref{lem:isolated_implies_Dp_empty}, the space $X$ has some point different from $a_{n+1}$; choose $x\in X$ with $x\neq a_{n+1}$.
Define a sequence $T=(T_k)$ by
\[
T_k=
\begin{cases}
a_k,& k\neq n+1,\\
x,& k=n+1.
\end{cases}
\]
Then $T$ agrees with $a$ for all sufficiently large $k$, so by Lemma~\ref{lem:finite_modification}(i) we have $T\in\Sp$.
Also $T_k=a_k$ for all $1\le k\le n$, hence $T\in A_n$.
Since $T_{n+1}=x\neq a_{n+1}$, we have $T\neq a$.
Thus $T\in B_n$.
\end{proof}

\begin{proposition}\label{prop:Bn_is_testset}
Assume $\Dp\neq\emptyset$.
Then for every $n\in\N$,
\[
B_n\in\F.
\]
\end{proposition}

\begin{proof}
Fix $n\in\N$ and let $f\in\Dp$.
By Lemma~\ref{lem:seq_char_continuity}, the witness family $D(f)$ is nonempty, so choose
\[
S=(S_k)\in D(f).
\]
By Remark~\ref{rem:nontrivial_space_from_Dp}, choose $x\in X$ with $x\neq a_{n+1}$.
Define a new sequence $T=(T_k)$ by
\[
T_k=
\begin{cases}
a_k,& 1\le k\le n,\\
x,& k=n+1,\\
S_k,& k\ge n+2.
\end{cases}
\]
Then $T$ differs from $S$ in only finitely many terms.
Hence by Lemma~\ref{lem:finite_modification},
\[
T\in\Sp \quad\text{and}\quad T\in D(f).
\]
Moreover, $T\in A_n$ because its first $n$ terms agree with those of $a$, while
$T\neq a$ because $T_{n+1}=x\neq a_{n+1}$.
Thus $T\in B_n\cap D(f)$.

Since $f\in\Dp$ was arbitrary, it follows that $B_n$ is a test set.
\end{proof}

\subsection{Existence of a bad maximal chain}

We now show that although each $B_n$ is a test set, their intersection is empty.

\begin{lemma}\label{lem:Bn_intersection_empty}
\[
\bigcap_{n\in\N} B_n=\emptyset.
\]
\end{lemma}

\begin{proof}
Suppose that $T\in\bigcap_{n\in\N} B_n$.
Then for every $n$, since $T\in B_n\subseteq A_n$, we have
\[
T_k=a_k \qquad\text{for all } 1\le k\le n.
\]
As this holds for every $n$, it follows that $T_k=a_k$ for all $k\in\N$.
Thus $T=a$.

But $T\in B_1=A_1\setminus\{a\}$ implies $T\neq a$, a contradiction.
Therefore
\[
\bigcap_{n\in\N} B_n=\emptyset.
\]
\end{proof}

At this stage the construction is complete.
We have produced a descending chain of test sets whose intersection is empty.
A maximal extension of this chain is therefore automatically bad.

\begin{theorem}\label{thm:bad_chain_exists}
Assume $\Dp\neq\emptyset$.
Then there exists a bad maximal chain in $(\F,\subseteq)$.
\end{theorem}

\begin{proof}
By Lemma~\ref{lem:Bn_chain} and Proposition~\ref{prop:Bn_is_testset}, the family
\[
\mathcal C_0:=\{B_n:n\in\N\}
\]
is a descending chain in $(\F,\subseteq)$.

By Lemma~\ref{lem:chain_extension}, there exists a maximal chain
\[
\mathcal C\subseteq\F
\]
such that
\[
\mathcal C_0\subseteq\mathcal C.
\]
Since $\mathcal C_0\subseteq\mathcal C$, we have
\[
\bigcap_{A\in\mathcal C}A \subseteq \bigcap_{n\in\N}B_n.
\]
By Lemma~\ref{lem:Bn_intersection_empty},
\[
\bigcap_{n\in\N}B_n=\emptyset.
\]
Hence
\[
\bigcap_{A\in\mathcal C}A=\emptyset.
\]
Because $\Dp\neq\emptyset$, the empty set is not a test set.
Therefore
\[
\bigcap_{A\in\mathcal C}A\notin\F.
\]
Thus $\mathcal C$ is a bad maximal chain.
\end{proof}

\section{Constructing a minimal test set}\label{sec:minimal}

Throughout this section we work under the standing framework of Section~\ref{sec:intro} at the fixed point $P$:
$(X,P)$ is Fr\'echet--Urysohn at $P$, $X$ is $T_1$ at $P$, $P$ is non-isolated, and $\Dp\neq\emptyset$.

The goal of this section is to construct a minimal test set.
As a formal consequence, we will also obtain a good maximal chain.
A direct construction at the level of sequences is difficult, so we first pass from injective convergent sequences to their ranges in $X\setminus\{P\}$.
This reduces the problem to a set-theoretic one, where the minimality question can be handled using almost disjointness.

\subsection{Recasting the problem in terms of ranges}

Because $X$ is $T_1$ at $P$, Lemma~\ref{lem:kernel_T1} gives $K_P=\{P\}$.
Thus the relevant region outside the limit point is simply $X\setminus\{P\}$.
We now isolate the subsets of $X\setminus\{P\}$ that behave like ranges of injective sequences converging to $P$.

\begin{definition}[The family $\mathcal I_P$]\label{def:IP_T1}
Define $\mathcal I_P$ to be the family of sets $M\subseteq X\setminus\{P\}$ such that:
\begin{enumerate}[label=(\roman*)]
\item $M$ is countably infinite;
\item for every neighborhood $U\in\mathcal N(P)$, the set $M\setminus U$ is finite.
\end{enumerate}
\end{definition}

The point of introducing $\mathcal I_P$ is that injective convergent sequences carry more information than we need.
For the argument below, what matters is not the order in which points appear, but only which points appear and whether they still accumulate at $P$.
The family $\mathcal I_P$ records exactly that information.

\begin{example}[$\mathcal I_P$ on $\R$]\label{ex:IP_real_T1}
Let $X=\R$ with the usual topology and $P=0$.
Then
\[
M=\{1/n:n\in\N\}\in\mathcal I_0.
\]
Indeed, $M$ is countably infinite, and for every neighborhood $U$ of $0$, all but finitely many points of $M$ lie in $U$.

This membership is a property of the set $M$, not of an arbitrary enumeration of it.
For instance, define a sequence $S=(S_k)$ by
\[
S_{2k-1}=1,
\qquad
S_{2k}=\frac{1}{k}
\qquad (k\in\N).
\]
Then $\Range(S)=M\in\mathcal I_0$, but $S$ does not converge to $0$, since the value $1$ occurs infinitely often.
Thus $\mathcal I_P$ records which points appear and how the set accumulates at $P$; it does not assert that every enumeration of the set converges to $P$.
\end{example}

\begin{lemma}[$\mathcal I_P$ is nonempty]\label{lem:IP_nonempty}
Assume $X$ is $T_1$ at $P$, $P$ is non-isolated, and $(X,P)$ is Fr\'echet--Urysohn at $P$.
Then $\mathcal I_P\neq\emptyset$.
\end{lemma}

\begin{proof}
By non-isolation, $P\in\cl{X\setminus\{P\}}$.
By the Fr\'echet--Urysohn property at $P$, choose a sequence $(x_n)\subseteq X\setminus\{P\}$ with $x_n\to P$.
Let
\[
M:=\{x_n:n\in\N\}\subseteq X\setminus\{P\}.
\]

We first show that $M$ is infinite.
If $M$ were finite, then some value $x\neq P$ would occur infinitely often in the sequence $(x_n)$.
Since $X$ is $T_1$ at $P$, there exists an open neighborhood $U\in\mathcal N(P)$ with $x\notin U$.
But $x_n\to P$ implies $x_n\in U$ eventually, a contradiction.
Thus $M$ is infinite; being the image of $\N$, it is countably infinite.

Let $U\in\mathcal N(P)$.
Since $x_n\to P$, there exists $N$ such that $x_n\in U$ for all $n\ge N$.
Hence every element of $M\setminus U$ appears among $\{x_1,\dots,x_{N-1}\}$, so $M\setminus U$ is finite.
Therefore $M\in\mathcal I_P$.
\end{proof}

The next lemma shows that every member of $\mathcal I_P$ really does come from an injective convergent sequence.
Thus $\mathcal I_P$ is not merely analogous to the collection of ranges of such sequences; it is exactly the right set-theoretic replacement.

\begin{lemma}[Enumeration lemma]\label{lem:enumerate_IP}
Assume $\mathcal I_P\neq\emptyset$.
For every $M\in\mathcal I_P$ there exists an injective sequence $T^M\in\Sp$ such that
\[
\Range(T^M)=M.
\]
\end{lemma}

\begin{proof}
Fix $M\in\mathcal I_P$.
Choose a bijection $e:\N\to M$ and set $m_n:=e(n)$.
Then $(m_n)$ is injective and has range $M$.

Let $U\ni P$ be a neighborhood.
Since $M\in\mathcal I_P$, the set $M\setminus U$ is finite.
Because $(m_n)$ enumerates each element of $M$ exactly once, it hits $M\setminus U$ only finitely many times.
Hence there exists $N$ such that $m_n\in U$ for all $n\ge N$.
Thus $(m_n)\to P$, so $T^M:=(m_n)\in\Sp$ and $\Range(T^M)=M$.
\end{proof}

We may therefore work interchangeably with injective convergent sequences and with their ranges in $\mathcal I_P$.
From this point on, the problem of constructing a minimal test set is reduced to finding the right family inside $\mathcal I_P$.

\subsection{Building a minimal test set from a MAD family}

We need a family inside $\mathcal I_P$ with two properties: maximality and almost disjointness.
Maximality will be used to prove the test-set property, and almost disjointness will be used to prove minimality.
We therefore choose a maximal almost disjoint family.
Almost disjoint families and their role in topology are surveyed by Hru\v{s}\'{a}k~\cite{Hru14}.

\begin{definition}[Almost disjointness $\perp_P$]\label{def:perpP}
For $A,B\subseteq X\setminus\{P\}$, we say that $A$ and $B$ are \emph{almost disjoint} if
\[
|A\cap B|<\infty.
\]
In this case, we write $A\perp_P B$.
\end{definition}

By non-redundancy, we mean that no member of the candidate test set can be removed while preserving the test-set property.
Almost disjointness will be used to prove this non-redundancy.
When one sequence is removed, we will choose a discontinuous function supported on its range.
Since every other range meets that range only finitely often, the remaining sequences will eventually fail to detect that function.
Thus each sequence becomes indispensable, and this is what will force minimality.

\begin{definition}[Maximal almost disjoint family]\label{def:MAD}
A family $\mathcal M\subseteq\mathcal I_P$ is called a \emph{maximal almost disjoint family}, or simply a \emph{MAD family}, if
\begin{enumerate}[label=(\roman*)]
\item for $M_1\neq M_2$ in $\mathcal M$, one has $M_1\perp_P M_2$;
\item for every $M\in\mathcal I_P$, there exists $N\in\mathcal M$ such that
\[
|M\cap N|=\infty.
\]
\end{enumerate}
\end{definition}

Condition (i) says that $\mathcal M$ is almost disjoint.
Condition (ii) is the maximality condition: no member of $\mathcal I_P$ can be added to $\mathcal M$ while preserving pairwise almost disjointness.
Indeed, if a witness range $M\in\mathcal I_P$ met every member of $\mathcal M$ only finitely often, then $M$ could be added to $\mathcal M$ without destroying almost disjointness.
This would contradict maximality.
Hence every witness range $M\in\mathcal I_P$ must meet some member of $\mathcal M$ infinitely often, which means that its witness role can be taken over by a member of the consequent test set.
This is exactly what will yield the test-set property.

\begin{lemma}\label{lem:MAD_exists}
Assume $\mathcal I_P\neq\emptyset$.
Then there exists a MAD family $\mathcal M\subseteq\mathcal I_P$.
\end{lemma}

\begin{proof}
We apply Zorn's lemma to the poset of pairwise almost disjoint subfamilies of $\mathcal I_P$, ordered by inclusion.
The key point is that the union of any chain of such families is again pairwise almost disjoint.

Let
\[
\mathfrak A:=\{\mathcal A\subseteq\mathcal I_P:\ \mathcal A\text{ is pairwise almost disjoint}\}
\]
be partially ordered by inclusion.
This poset is nonempty, since $\emptyset\in\mathfrak A$.

Let $\{\mathcal A_i\}_{i\in I}$ be a chain in $\mathfrak A$, and set
\[
\mathcal A:=\bigcup_{i\in I}\mathcal A_i.
\]
If $M\neq N$ are in $\mathcal A$, then $M\in\mathcal A_i$ and $N\in\mathcal A_j$ for some $i,j$.
Since the $\mathcal A_i$ form a chain under inclusion, one is contained in the other, so $M$ and $N$ lie in a common pairwise almost disjoint family.
Hence $|M\cap N|<\infty$, and therefore $\mathcal A\in\mathfrak A$.
By Zorn's lemma, $\mathfrak A$ has a maximal element $\mathcal M$.

Let $S\in\mathcal I_P$.
If $|S\cap M|<\infty$ for all $M\in\mathcal M$, then $\mathcal M\cup\{S\}$ is still pairwise almost disjoint, contradicting maximality.
Hence some $M\in\mathcal M$ satisfies $|S\cap M|=\infty$.
\end{proof}

This step is nonconstructive in general: the family $\mathcal M$ is obtained by Zorn's lemma, so the resulting minimal test set is not usually given explicitly.
The sequential fan in Section~\ref{subsec:sequential_fan} will provide a concrete case where the MAD family, and hence $A_{\min}$, can be written down.

We now pass back from sets to sequences.
The candidate minimal test set is obtained by choosing one injective convergent sequence for each member of the MAD family.

\begin{definition}[$A_{\min}$]\label{def:Amin_T1}
Fix a MAD family $\mathcal M\subseteq\mathcal I_P$.
For each $M\in\mathcal M$ choose an injective sequence $T^M\in\Sp$ with $\Range(T^M)=M$.
Define
\[
A_{\min}:=\{T^M:M\in\mathcal M\}\subseteq\Sp.
\]
\end{definition}

At this point the construction has produced a candidate family.
The next step is to check that it still detects every discontinuity, and the one after that is to check that none of its members is redundant.

\begin{lemma}\label{lem:Amin_is_testset}
Assume the standing framework of Section~\ref{sec:minimal}.
Then $A_{\min}$ is a test set, hence $A_{\min}\in\F$.
\end{lemma}

\begin{proof}
Let $f\in\Dp$.
Choose a witnessing sequence
\[
S=(S_n)\in D(f).
\]
Since $S\in D(f)$, there exists $\e>0$ such that the set
\[
I:=\{n\in\N:\ |f(S_n)-f(P)|\ge\e\}
\]
is infinite.
Choose strictly increasing indices $(n_k)_{k\in\N}\subseteq I$ and set
\[
R_k:=S_{n_k}.
\]
Let
\[
R:=\{R_k:k\in\N\}.
\]
We claim that $R\in\mathcal I_P$.

Since $R$ is the image of $\N$, it is at most countable.
If $R$ were finite, then $(R_k)$ would take values in a finite set $F$.
For each $x\in F$, choose an open neighborhood $U_x\in\mathcal N(P)$ with $x\notin U_x$.
Then
\[
U:=\bigcap_{x\in F}U_x
\]
is a neighborhood of $P$ disjoint from $F$.
But $(R_k)\to P$, so eventually $R_k\in U$, a contradiction.
Hence $R$ is infinite and therefore countably infinite.

Let $U\in\mathcal N(P)$.
Since $(R_k)\to P$, all but finitely many $R_k$ lie in $U$, so $R\setminus U$ is finite.
Thus $R\in\mathcal I_P$.

By Definition~\ref{def:MAD}(ii), choose $M\in\mathcal M$ with $|R\cap M|=\infty$.
Let $T^M\in A_{\min}$ enumerate $M$ injectively.
Since $R\cap M$ is infinite and $T^M$ visits each point of $M$ exactly once, the sequence $T^M$ passes through points of $R$ infinitely many times.
For those terms we have
\[
|f(T^M_j)-f(P)|\ge \e.
\]
Hence $f(T^M_n)\not\to f(P)$, so $T^M\in D(f)$.
Thus $A_{\min}\cap D(f)\neq\emptyset$.

Since $f\in\Dp$ was arbitrary, $A_{\min}$ is a test set.
Therefore $A_{\min}\in\F$.
\end{proof}

\begin{lemma}\label{lem:Amin_minimal}
Assume the standing framework of Section~\ref{sec:minimal}.
Then $A_{\min}$ is a minimal test set.
\end{lemma}

\begin{proof}
To prove minimality, remove one element of $A_{\min}$ and construct a discontinuous function that is detected by the removed sequence but by no remaining one.

Let $A'\subsetneq A_{\min}$.
Then $T^{M_0}\notin A'$ for some $M_0\in\mathcal M$.
Define $h:X\to\R$ by
\[
h(x)=
\begin{cases}
1, & x\in M_0,\\
0, & x\notin M_0.
\end{cases}
\]
Since $P\notin M_0$, we have $h(P)=0$.

Every term of $T^{M_0}$ lies in $M_0$, so
\[
h(T^{M_0}_n)\equiv 1\not\to 0=h(P).
\]
Thus $T^{M_0}\in D(h)$, hence $h\in\Dp$.

We claim that $A'\cap D(h)=\emptyset$.
Let $T^M\in A'$ with $M\neq M_0$.
By almost disjointness, $M\cap M_0$ is finite.
Since $T^M$ enumerates $M$ injectively, it visits the finite set $M\cap M_0$ only finitely many times.
Hence for all sufficiently large $n$ we have $T^M_n\notin M_0$, so
\[
h(T^M_n)=0
\]
eventually, and therefore $h(T^M_n)\to 0=h(P)$.
Thus $T^M\notin D(h)$. $A'\cap D(h)=\emptyset$.

Thus $A'$ fails to meet the witness family of the discontinuous function $h$, so $A'$ is not a test set.

We have shown that every proper subset of $A_{\min}$ fails to be a test set.
Therefore $A_{\min}$ is minimal in $(\F,\subseteq)$.
\end{proof}

The first question of the paper is therefore settled:
the poset $(\F,\subseteq)$ contains a minimal element.
The good maximal chain now follows directly from the minimal test set.

\begin{corollary}\label{cor:good_chain_exists_T1}
Assume the standing framework of Section~\ref{sec:minimal}.
Then there exists a good maximal chain in $(\F,\subseteq)$.
\end{corollary}

\begin{proof}
By Lemmas~\ref{lem:Amin_is_testset} and \ref{lem:Amin_minimal}, the set $A_{\min}$ is a minimal element of $\F$.
By Lemma~\ref{lem:chain_extension}, extend the chain $\{A_{\min}\}$ to a maximal chain $\mathcal C\subseteq\F$.
Then $A_{\min}\in\mathcal C$, so Lemma~\ref{lem:min_chain_intersection} gives
\[
\bigcap_{A\in\mathcal C}A=A_{\min}\in\F.
\]
Thus $\mathcal C$ is good.
\end{proof}

Together with Theorem~\ref{thm:bad_chain_exists}, Corollary~\ref{cor:good_chain_exists_T1} proves the two chain behaviors asserted in Theorem~\ref{thm:chains_exist_intro}.

\subsection{The sequential fan as a concrete model of cardinal efficiency}\label{subsec:sequential_fan}

After the general existence and chain constructions, it is useful to examine a concrete space in which the abstract objects can be computed explicitly.
The point of the following example is not merely to illustrate the constructions above, but to show that the method developed in Section~\ref{sec:minimal} can also yield cardinal efficiency.
Although a minimal test set need not have smaller cardinality than the full family $\Sp$, the sequential fan provides a natural space in which the resulting minimal test set is strictly smaller than $\Sp$.

We use the standard countable sequential fan.
It is obtained by taking the topological sum of countably many convergent sequences and identifying all of their limit points to a single point; see Franklin--Smith Thomas~\cite{FranklinSmithThomas1977}.
We denote the resulting quotient space by $S_\omega$.

Concretely, we realize $S_\omega$ as
\[
S_\omega:=\{P\}\cup\{x_{n,m}:n,m\in\N\}.
\]
For each $n\in\N$, write
\[
B_n:=\{x_{n,m}:m\in\N\}
\]
for the $n$th spoke.
Each point $x_{n,m}$ is isolated, and the distinguished point $P$ is the common limit point obtained by identifying the limit points of the countably many convergent sequences.
A neighborhood base at $P$ is given by
\[
U_f:=\{P\}\cup\{x_{n,m}:m\ge f(n)\},
\]
where $f:\N\to\N$ is arbitrary.
Thus a neighborhood of $P$ contains, in each spoke $B_n$, a tail of the sequence
\[
x_{n,1},x_{n,2},x_{n,3},\dots .
\]
Equivalently, for each fixed $n$, the canonical sequence $(x_{n,m})_{m\in\N}$ converges to $P$ in $S_\omega$.

\begin{figure}[ht!]
\centering
\begin{tikzpicture}[x=1cm,y=1cm,>=Latex]
  \coordinate (P) at (0,0);
  \fill[red] (P) circle (2.2pt);
  \node[below left] at (P) {$P$};

  \foreach \ang/\lab/\pos in {155/$B_1$/above, 95/$B_2$/right, 30/$B_3$/right, -35/$B_4$/right, -95/$B_5$/right} {
    \draw[thick] ($(P)+(\ang:1.12)$) -- ($(P)+(\ang:4.55)$);
    \draw[red, very thick] (P) -- ($(P)+(\ang:1.12)$);
    \node[\pos] at ($(P)+(\ang:4.55)$) {\lab};
  }

  \foreach \r/\s in {3.95/2.3,3.25/2.0,2.70/1.8,2.25/1.6,1.95/1.45,1.72/1.35,1.53/1.25,1.38/1.18,1.24/1.10} {
    \fill ($(P)+(155:\r)$) circle (\s pt);
  }
  \node[above] at ($(P)+(155:3.95)$) {$x_{1,1}$};
  \node[above] at ($(P)+(155:3.25)$) {$x_{1,2}$};
  \node[above] at ($(P)+(155:2.65)$) {$x_{1,3}$};

  \foreach \r/\s in {3.95/2.3,3.25/2.0,2.70/1.8,2.25/1.6,1.95/1.45,1.72/1.35,1.53/1.25,1.38/1.18,1.24/1.10} {
    \fill ($(P)+(95:\r)$) circle (\s pt);
  }
  \node[right] at ($(P)+(95:3.95)$) {$x_{2,1}$};
  \node[right] at ($(P)+(95:3.25)$) {$x_{2,2}$};
  \node[right] at ($(P)+(95:2.70)$) {$x_{2,3}$};

  \foreach \r/\s in {3.95/3.2,3.25/2.8,2.70/2.3,2.25/1.8,1.95/1.55,1.72/1.40,1.53/1.28,1.38/1.18,1.24/1.10} {
    \fill ($(P)+(30:\r)$) circle (\s pt);
  }
  \node[right] at ($(P)+(29:3.95)$) {$x_{3,1}$};
  \node[right] at ($(P)+(29:3.25)$) {$x_{3,2}$};
  \node[right] at ($(P)+(29:2.70)$) {$x_{3,3}$};

  \foreach \r/\s in {3.95/3.2,3.25/2.8,2.70/2.3,2.25/1.8,1.95/1.55,1.72/1.40,1.53/1.28,1.38/1.18,1.24/1.10} {
    \fill ($(P)+(-35:\r)$) circle (\s pt);
  }
  \node[right] at ($(P)+(-35:3.95)$) {$x_{4,1}$};
  \node[right] at ($(P)+(-35:3.25)$) {$x_{4,2}$};
  \node[right] at ($(P)+(-35:2.70)$) {$x_{4,3}$};

  \foreach \r/\s in {3.95/2.3,3.25/2.0,2.70/1.8,2.25/1.6,1.95/1.45,1.72/1.35,1.53/1.25,1.38/1.18,1.24/1.10} {
    \fill ($(P)+(-95:\r)$) circle (\s pt);
  }
  \node[right] at ($(P)+(-95:2.70)$) {$x_{5,3}$};
  \node[right] at ($(P)+(-95:3.25)$) {$x_{5,2}$};
  \node[right] at ($(P)+(-95:3.95)$) {$x_{5,1}$};

  \draw[black!55, thin] ($(P)+(-118:0.55)$) -- ($(P)+(-118:4.15)$);
  \draw[black!55, thin] ($(P)+(-142:0.55)$) -- ($(P)+(-142:3.95)$);

  \draw[red,->,thick] (1.05,2.0) -- (0.15,0.18);
  \node[red, right] at (1.08,2.0) {$U_f\cap B_n$};
  \node[right] at (2.0,-1) {$T_4=(x_{4,k})_{k\in\N}$};
\end{tikzpicture}
\caption{The countable sequential fan $S_\omega$. Each spoke is $B_n=\{x_{n,m}:m\in\N\}$ and only five labeled spokes are shown explicitly; the fan has countably many spokes. The red segment near $P$ represents the tail $U_f\cap B_n$ of a basic neighborhood $U_f$. The canonical sequence $T_n=(x_{n,k})_{k\in\N}$ along each spoke is an element of the minimal test set.}
\label{fig:sequential_fan}
\end{figure}
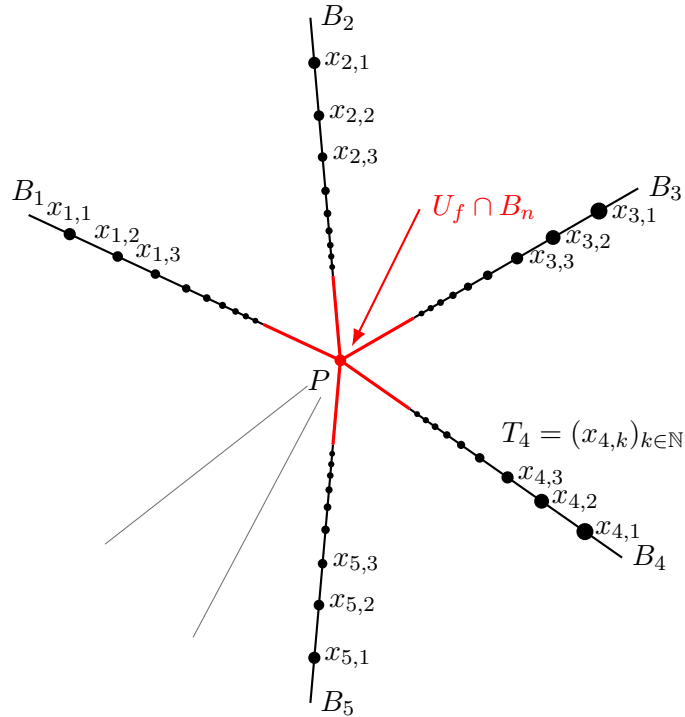

The space $S_\omega$ satisfies the standing framework.
Indeed, it is $T_1$ because every $x_{n,m}$ is isolated, the point $P$ is non-isolated by construction, and the characteristic function of $\{P\}$ is discontinuous at $P$.
To see that $(S_\omega,P)$ is Fr\'echet--Urysohn at $P$, let $A\subseteq S_\omega\setminus\{P\}$ with $P\in\cl{A}$.
If each $A\cap B_n$ were finite, we could define a function $f:\N\to\N$ by setting $f(n)=1$ when $A\cap B_n=\emptyset$ and
\[
f(n)>\max\{m:x_{n,m}\in A\cap B_n\}
\]
otherwise.
Then $U_f\cap A=\emptyset$, a contradiction.
Hence $A\cap B_{n_0}$ is infinite for some $n_0$, and an increasing enumeration $x_{n_0,m_k}$ of points from $A\cap B_{n_0}$ satisfies $x_{n_0,m_k}\to P$.

In this space, the family $\mathcal I_P$ can be computed explicitly.

\begin{proposition}\label{prop:sequential_fan_IP}
Let $M\subseteq S_\omega\setminus\{P\}$. Then $M\in\mathcal I_P$ if and only if $M$ is countably infinite and contained in the union of finitely many spokes.
\end{proposition}

\begin{proof}
Assume first that $M\in\mathcal I_P$. If $M$ met infinitely many spokes, choose an infinite set $J\subseteq\N$ such that $M\cap B_n\neq\emptyset$ for each $n\in J$, and pick a point $x_{n,m_n}\in M\cap B_n$. Define a function $f:\N\to\N$ by
\[
f(n)=
\begin{cases}
m_n+1,& n\in J,\\
1,& n\notin J.
\end{cases}
\]
Then $x_{n,m_n}\in M\setminus U_f$ for every $n\in J$, so $M\setminus U_f$ is infinite, contradicting $M\in\mathcal I_P$. Hence $M$ is contained in the union of finitely many spokes.

Conversely, suppose that $M$ is countably infinite and contained in $B_{n_1}\cup\cdots\cup B_{n_k}$. Let $U_f$ be a neighborhood of $P$. For each $i$, the set $B_{n_i}\setminus U_f$ consists of the finitely many points $x_{n_i,m}$ with $m<f(n_i)$. Therefore
\[
M\setminus U_f\subseteq \bigcup_{i=1}^k (B_{n_i}\setminus U_f)
\]
is finite, so $M\in\mathcal I_P$.
\end{proof}

This description makes it easy to write down a countable MAD family in $\mathcal I_P$.

\begin{proposition}\label{prop:sequential_fan_MAD}
The family
\[
\mathcal M:=\{B_n:n\in\N\}
\]
is a MAD family in $\mathcal I_P$.
\end{proposition}

\begin{proof}
The sets $B_n$ are pairwise disjoint, hence pairwise almost disjoint. Let $M\in\mathcal I_P$. By Proposition~\ref{prop:sequential_fan_IP}, the set $M$ is contained in the union of finitely many spokes. Since $M$ is infinite, the pigeonhole principle gives some $n\in\N$ for which $|M\cap B_n|=\infty$. Thus $M$ meets some member of $\mathcal M$ in an infinite set, proving maximality.
\end{proof}

For each $n\in\N$, let
\[
T_n(k):=x_{n,k}
\qquad (k\in\N).
\]
Then $T_n\to P$, and the family
\[
A_{\min}^{\mathrm{fan}}:=\{T_n:n\in\N\}
\]
is precisely the minimal test set associated with the MAD family $\mathcal M$ from Proposition~\ref{prop:sequential_fan_MAD}.

\begin{corollary}\label{cor:sequential_fan_minimal}
In the sequential fan $S_\omega$, the family $A_{\min}^{\mathrm{fan}}=\{T_n:n\in\N\}$ is a minimal test set.
Moreover,
\[
|A_{\min}^{\mathrm{fan}}|=\aleph_0
\qquad\text{and}\qquad
|\Sp|=2^{\aleph_0}.
\]
In particular, $|A_{\min}^{\mathrm{fan}}|<|\Sp|$.
\end{corollary}

\begin{proof}
By Proposition~\ref{prop:sequential_fan_MAD} and the general construction of Definition~\ref{def:Amin_T1}, the family $A_{\min}^{\mathrm{fan}}$ is a minimal test set.
Since $A_{\min}^{\mathrm{fan}}=\{T_n:n\in\N\}$, it is countable.
On the other hand, $\Sp\subseteq S_\omega^\N$ and $S_\omega$ is countable, so $|\Sp|\le 2^{\aleph_0}$.
For the reverse inequality, fix the spoke $B_1$.
For each strictly increasing function $a:\N\to\N$, define a sequence $T^a$ by
\[
T^a_k=x_{1,a(k)}.
\]
Each $T^a$ belongs to $\Sp$, and distinct functions give distinct sequences.
Since there are $2^{\aleph_0}$ strictly increasing functions $\N\to\N$, we get $|\Sp|\ge 2^{\aleph_0}$.
Thus $|\Sp|=2^{\aleph_0}$.
\end{proof}

Thus the sequential fan shows that the construction of Section~\ref{sec:minimal} can realize genuine cardinal compression in a natural space:
the resulting minimal test set may be strictly smaller than the full family $\Sp$.

\subsection{A discussion for non-\texorpdfstring{$T_1$}{T1} spaces}\label{subsec:nonT1_spaces}

If $T_1$ fails at $P$, the kernel
\[
K_P=\bigcap_{U\in\mathcal N(P)}U
\]
may satisfy $K_P\supsetneq\{P\}$.
In this situation, points of $K_P$ give rise to constant sequences converging to $P$:
if $x\in K_P$, then the constant sequence
\[
c_x:=(x,x,x,\dots)
\]
satisfies $c_x\to P$, since $x\in U$ for every neighborhood $U\in\mathcal N(P)$.

These constant sequences immediately detect a certain class of discontinuities.
Indeed, if $x\in K_P\setminus\{P\}$ and $f:X\to\R$ satisfies $f(x)\neq f(P)$, then along $c_x$ we have
\[
f(c_{x,n})\equiv f(x)\neq f(P),
\]
so $f(c_{x,n})\not\to f(P)$.
Hence $c_x\in D(f)$ and $f\in\Dp$.

However, constant sequences alone need not detect all discontinuities in $\Dp$.
It may happen that $f|_{K_P}\equiv f(P)$ while $f$ is still discontinuous at $P$, with the discontinuity witnessed only by sequences that eventually lie in $X\setminus K_P$.
Accordingly, in the non-$T_1$ regime one should distinguish discontinuities already visible on $K_P$
from those arising from points in $X\setminus K_P$.

\section*{Acknowledgments}

The author would like to thank Professor Sanghoon Kwak for his guidance,
encouragement, and helpful comments on earlier versions of this paper.
The author also thanks Professor Inseok Seo for an inspiring course on set theory
and mathematical logic.


\begin{thebibliography}{99}

\bibitem[Eng89]{Eng89}
R.~Engelking,
\emph{General Topology},
2nd ed., Sigma Series in Pure Mathematics, Vol.~6,
Heldermann Verlag, Berlin, 1989.

\bibitem[Fra65]{Fra65}
S.~P.~Franklin,
Spaces in which sequences suffice,
\emph{Fundamenta Mathematicae} \textbf{57} (1965), 107--115.

\bibitem[FST77]{FranklinSmithThomas1977}
S.~P.~Franklin and B.~V.~Smith Thomas,
On the metrizability of \(k_{\omega}\)-spaces,
\emph{Pacific Journal of Mathematics} \textbf{72} (1977), no.~2, 399--402.

\bibitem[Hru14]{Hru14}
M.~Hru\v{s}\'{a}k,
Almost disjoint families and topology,
in \emph{Recent Progress in General Topology III}
(K.~P.~Hart, J.~van Mill, P.~Simon, eds.),
Atlantis Press, Paris, 2014, 601--638.

\bibitem[Jec03]{Jec03}
T.~Jech,
\emph{Set Theory: The Third Millennium Edition, revised and expanded},
Springer Monographs in Mathematics,
Springer, Berlin--Heidelberg, 2003.

\end{thebibliography}
\end{document}